\documentclass[12pt]{amsart}

\usepackage{a4,amsmath,amssymb}
\usepackage{verbatim}

\usepackage{tikz}
\usetikzlibrary{decorations.markings}

\usepackage{enumitem}

\setlist{itemsep=1mm,topsep=1mm}
\setitemize{leftmargin=7mm,labelsep=2mm}

\theoremstyle{plain}
\newtheorem{thm}{Theorem}[section]
\newtheorem{lemma}[thm]{Lemma}
\newtheorem{prop}[thm]{Proposition}
\newtheorem{cor}[thm]{Corollary}

\theoremstyle{definition}
\newtheorem{definition}[thm]{Definition}
\newtheorem{question}[thm]{Question}

\newtheorem*{remark*}{Remark}

\DeclareMathOperator{\Av}{Av}

\newcommand{\digraphsclass}{\mathcal{D}}
\newcommand{\Rdigraphsclass}{\digraphsclass_R}
\newcommand{\Idigraphsclass}{\digraphsclass_I}
\newcommand{\Rgraphsclass}{\mathcal{G}_R}
\newcommand{\Igraphsclass}{\mathcal{G}_I}
\newcommand{\Rtournamentsclass}{\mathcal{T}_R}

\begin{document}

\title[WQO, graphs and homomorphic image orderings]{On well quasi-order of graph classes under homomorphic image orderings}

\author{S. Huczynska}
\address{School of Mathematics and Statistics, University of St Andrews, St Andrews, Scotland, U.K.}
\email{sh70@st-andrews.ac.uk}

\author{N. Ru\v{s}kuc}
\address{School of Mathematics and Statistics, University of St Andrews, St Andrews, Scotland, U.K.}
\email{nik.ruskuc@st-andrews.ac.uk}

\begin{abstract}
In this paper we consider the question of well quasi-order for classes defined by a single obstruction within the classes of all graphs, digraphs and tournaments, under the homomorphic image ordering (in both its standard and strong forms).  The homomorphic image ordering was introduced by the authors in a previous paper and corresponds to the existence of a surjective homomorphism between two structures. We obtain complete characterizations in all cases except for graphs under the strong ordering, where some open questions remain.
\end{abstract}


\maketitle

\section{Introduction}
\label{sec-intro}

Combinatorial structures have been considered under various different orderings; for example, substructure order (for which we may make a further distinction between weak and induced) and homomorphism order.  For specific types of combinatorial object, there are other well-known orderings, for example the minor order on the class of graphs.  All of these have received considerable attention in the combinatorial literature.

The starting premise in this work is the observation that the notion of homomorphism provides a useful unifying viewpoint from which to consider many of these orderings. 
Two structures $A$ and $B$ are related under the homomorphism (quasi-)order if there exists \emph{any} homomorphism between them, while $A$ and $B$ are related under the substructure order if there exists an \emph{injective} homomorphism between them (a ``standard" homomorphism in the case of weak substructure, and a strong homomorphism in the case of induced substructure).  
The study of substructure orderings is pervasive throughout combinatorics; for an introduction into the homomorphism ordering for graphs the reader may refer to \cite[Chapter 3]{hell}.

By way of analogy with the substructure order, it is natural to consider the partial order corresponding to the existence of a \emph{surjective} homomorphism between two structures.  As with the substructure order, we may distinguish between weak and induced forms.   In a previous paper (\cite{HucRus15}), we introduced this order, which we called the \emph{homomorphic image order}; consideration of different strengths led to the standard, strong and $M$-strong forms of the order.  Perhaps surprisingly, this order had previously received very little attention in the literature.  
One notable exception is \cite{landraitis}, where the homomorphic image ordering is considered for the class of countable linear orders.   
As well as the naturalness of the definition, another motivation for studying 
the homomorphic image orders is that the graph minor order may be viewed as a composition of substructure order with a special kind of homomorphic image order.
With the recent increase in prominence of minor-like orders (see, for example, 
\cite{Bla}, \cite{kim15}),
  one might hope that further understanding of homomorphic image orders for graphs could enable new insights into the minor orders.

Some fundamental graph-theoretical properties are preserved by the taking of homomorphic images, for instance being connected, and having diameter at most $d$.
In particular, the class of all connected graphs can be defined by avoiding (in the sense of homomorphic images) the empty graph of size $2$.
It is perhaps worth noting that the above properties are not preserved by taking subgraphs (standard or induced), while the properties that are known to be preserved by the latter, such as planarity, are not preserved by homomorphic images.

This duality carries through into the area of well quasi-order and antichains; it transpires that the properties of the homomorphic image order are quite different in flavour from those of the more familiar substructure order.  Within the class of (reflexive) graphs, many of the ``classic'' antichains under the substructure order, for example cycles and double-ended forks, are not antichains under the homomorphic image order (both of these in fact become chains).  Conversely, antichains under the homomorphic image order may not be antichains in the substructure order; for example, the family of complete graphs with alternate perimeter edges deleted, forms an antichain under the former but not under the latter.  

Well quasi-order for classes of graphs and related combinatorial structures is a natural and much-studied topic.  Whenever we have classes of such structures which we wish to compare, for example in terms of inclusion or homomorphic images, we are led to consider downward-closed sets under the chosen orderings.  The concept of well quasi-order then allows us to distinguish between what we may call (following Cherlin in \cite{Che11}) ``tame" and ``wild" such classes.

A \emph{quasi-order} is a binary relation which is reflexive ($x \leq x$ for all $x$) and transitive ($x \leq y \leq z$ implies $x \leq z$).  A quasi-order which is also anti-symmetric ($x \leq y \leq x$ implies $x=y$) is called a \emph{partial order}; all orders considered in this paper are partial orders.

A \emph{well quasi-order (wqo)} is  a quasi-order which is \emph{well-founded}, i.e. every strictly decreasing sequence is finite, and \emph{has no infinite antichain}, i.e. every set of pairwise incomparable elements is finite.  Since we will be considering only finite structures, and our orderings respect size, wqo is equivalent to the non-existence of infinite antichains throughout.

Given a quasi-order $(X,\leq)$, a subset $I$ of $X$ is called an \emph{ideal} or \emph{downward closed set} if $y \leq x \in I$ implies $y \in I$.  Ideals are precisely \emph{avoidance sets}, i.e. sets of the form $\Av(B)=\{x \in X : (\forall b \in B)(b \not \leq x) \}$.  Here $B$ is an arbitrary subset of $X$, finite or infinite.  The situation in which every ideal is defined by a \emph{finite} avoidance set is precisely the case when $X$ is wqo.

Questions about well quasi-order of graphs and related structures have been extensively investigated.  While the class of all graphs is not wqo under the subgraph order nor the induced subgraph order, a celebrated result of Robertson and Seymour (\cite{RobSey04}) establishes that it is wqo under the minor order.  When a class itself is not wqo, one can investigate the wqo ideals within that class and attempt to describe them.  For example, a result due to Ding (\cite{Din92}) establishes that an ideal of graphs with respect to the subgraph ordering is wqo precisely if it contains only finitely many cycles and double-ended forks.

We may ask the following general question about a class $(\mathcal{C}, \leq)$ of finite structures equipped with a natural ordering:  \emph{given a finite set  $\{ X_1,\dots,X_k\} \subseteq \mathcal{C}$ of forbidden structures, is the ideal $\Av(X_1,\dots,X_k)$ wqo?} 

It is easy to see that Ding's result (\cite{Din92}) resolves this question for subgraph ordering.
For the induced subgraph order the situation is much more complicated, and indeed the general wqo question remains open.
In the case of a single obstruction, Damaschke (\cite{Dam90}) proved that $\Av(G)$ is wqo if and only if $G$ is an induced subgraph of the path on $4$ vertices.
Some progress is made on classes defined by two obstructions in \cite{Kor11a}.
Similar analyses have been undertaken for some specific classes of graphs, such as bipartite graphs (\cite{Kor11}) and permutation graphs  (\cite{Atm}), defined by a small number of obstructions.
In a recent article \cite{Bla} the induced minor ordering is considered, where induced subgraph replaces subgraph in the usual minor definition; yet again, the class of all graphs is not wqo under this ordering, and a classification is obtained for wqo classes defined by a single obstruction.
Finally, in the class of all finite tournaments under the subtournament order, an ideal $\Av(T)$ is wqo if and only if $T$ is a linear tournament or one of three small exceptions (of size $5,6,6$, respectively); see \cite{Lat94}.  The general wqo question for arbitrary ideals of the form $\Av(T_1,\ldots,T_k)$ is wide open.

For a general discussion of wqo in a variety of combinatorial settings, we refer the reader to \cite{HucRus15a}.
The survey article
by Cherlin  (\cite{Che11}) specifically considers the wqo question in the setting of graphs, tournaments and permutations under substructure order, and discusses its algorithmic  aspects.

In \cite{HucRus15}, where a systematic study of homomorphic image orders was initiated, we have shown that the homomorphic image orders are not wqo within the classes of all graphs, 
digraphs and tournaments (although standard and strong are wqo for trees).
In line with the situation for embedding orderings, briefly outlined above, the next natural step is to consider the wqo question for proper subclasses.
In this paper we consider the subclasses defined by a single obstruction under the homomorphic image ordering and the strong homomorphic image ordering.  We obtain complete characterizations in all instances except for graphs under the strong ordering, where some open questions remain.

\section{Preliminaries}
\label{sec-prelim}

In \cite{HucRus15}, we introduced the homomorphic image order for arbitrary relational structures.  Since, in this paper, we consider only graph-like structures, it is sufficient to give the definitions for the case of finite structures with a single binary relation.

\begin{definition}
For two structures $\mathcal{S}=(S,R_S)$ and $\mathcal{T}=(T,R_T)$, 
with $R_S$ and $R_T$ binary, and a mapping $\phi:S\rightarrow T$, we let
\[
\phi(R_S)=\{ (\phi(s_1),\phi(s_2)) : (s_1,s_2)\in R_S\},
\]
and say that $\phi$ is:
\begin{itemize}
\item[(i)] a (\emph{standard}) \emph{homomorphism} if $(s_1,s_2) \in R_S \Rightarrow (\phi(s_1),\phi(s_2)) \in R_T$,  i.e., if $\phi(R_S)\subseteq R_T|_{\phi(S)}$;
\item[(ii)] a \emph{strong homomorphism} if $\phi$ is a homomorphism
and $\phi(R_S)=R_T|_{\phi(S)}$.
\end{itemize}
A surjective (strong) homomorphism is called a (strong) \emph{epimorphism}.
\end{definition}

Our definition of strong homomorphism requires that every related pair in $\phi(S)$ must be the image of at least one related pair in $S$.

\begin{definition}
For a class $\mathcal{C}$ of relational structures, we define two orders on its members as follows:
\begin{itemize}
\item homomorphic image order: for $A,B\in\mathcal{C}$, $A \preceq B$ if there exists an epimorphism $B \rightarrow A$;
\item strong (induced) homomorphic image order: for $A,B\in\mathcal{C}$, $A \preceq B$ if there exists a strong epimorphism $B \rightarrow A$.
\end{itemize}
\end{definition}

It was shown in \cite{HucRus15} that both these relations are actually partial orders if $\mathcal{C}$ consists of finite structures.

We now proceed to define the structures which we will consider.

\begin{definition}  
A \emph{digraph} is a set $D$ with a binary relation $E(D)$.
\end{definition}

In a digraph $D$, a related pair $(x,y) \in E(D)$ is called a \emph{(directed) edge}.  Sometimes we will write $x \rightarrow y$ to indicate that $(x,y) \in E(D)$
and $x||y$ to mean $x\not\rightarrow y$ and $y\not\rightarrow x$.  
A digraph $D$ is said to be \emph{reflexive} if every pair $(x,x)$ ($x\in D$) is an edge, and \emph{irreflexive} if no such pair is an edge.
Note that being irreflexive is different (and stronger) than not being reflexive.
A digraph $D$ is said to be \emph{complete} if all pairs $(x,y)$, with $x$ and $y$ distinct, are edges; and it is said to be \emph{empty} if there is no edge $(x,y)$ with $x\neq y$. 

A digraph homomorphism $\phi:D_1 \rightarrow D_2$ maps edges to edges; $\phi$ is strong if it maps $E(D_1)$ \emph{onto} $E(D_2)$.

\begin{definition}
A \emph{graph} is a digraph $G$ in which the edge relation $E(G)$ is symmetric.
\end{definition}

Here, an (undirected) edge corresponds to two pairs $(x,y)$ and $(y,x)$ and we often denote this by $\{x,y\}$.  Furthermore, we will require that $E(G)$ is either irreflexive or reflexive.  This choice affects the notion of homomorphisms: in the irreflexive version a homomorphism may not ``collapse" an edge to a single vertex, while in the reflexive version both edges and non-edges may be so collapsed.

\begin{definition}
A \emph{tournament} is a digraph $T$ in which, for any two distinct $x,y \in T$, precisely one of $(x,y)$ or $(y,x)$ is an edge.  
\end{definition}

Again, we consider reflexive and irreflexive tournaments.  In the irreflexive case, since a homomorphism may not collapse an edge, every homomorphism is injective.

When there is no risk of confusion, we will notationally identify a structure with the set of its elements.

We now proceed to prove a technical wqo result and some consequences which will be repeatedly used throughout the paper.

\begin{thm}\label{HigmanDigraph}
\label{thmTHigmanD}
Let $N$ be a natural number, and let $\mathcal{T}_{N}$ be the class of all digraphs $D$ whose vertex set can be split into a disjoint union $D_e\cup D_c\cup D_f$ such that:
\begin{itemize}
\item
the digraph induced on $D_e$ is empty (reflexive or irreflexive);
\item
the digraph induced on $D_c$ is complete (reflexive);
\item
$|D_f|\leq N$;
\item
the connections between $D_e$ and $D_c$ are uniform, in the sense that for all $x,y\in D_e$ and all $z,t\in D_c$ we have:
\begin{eqnarray*} 
&&
x\rightarrow z \Leftrightarrow y\rightarrow t, \mbox{ and } 
\\
&&
 x\leftarrow z \Leftrightarrow y\leftarrow t.
\end{eqnarray*}
\end{itemize}
The class $\mathcal{T}_{N}$ is well quasi-ordered under both the standard and strong homomorphic image orderings.
\end{thm}

\begin{remark*}
Intuitively, the uniformity requirement between $D_e$ and $D_c$ can be viewed as follows:
either all the pairs $(x,z)$ with $x\in D_e$ and $z\in D_c$ are edges of $D$, or none are;
likewise, either all the pairs $(z,x)$ with $x\in D_e$ and $z\in D_c$ are edges of $D$, or none are.
\end{remark*}

\begin{proof}
It suffices to prove the result for the strong homomorphic image ordering.
Suppose $\mathcal{A} \subseteq \mathcal{T}_N$ is an infinite antichain. 
Since
\begin{itemize}
\item there are finitely many digraphs of size $\leq N$;
\item there are four possible connections between $D_e$ and $D_c$;
\item $\mathcal{A}$ is infinite;
\end{itemize}
we may assume without loss of generality that all $D \in \mathcal{A}$ have the same $D_f$ and the same type of connection between $D_e$ and $D_c$.  Write $F$ for the (common) $D_f$.

In what follows we use the following observations: (1) There are finitely many ways in which a vertex from $D_e$ or $D_c$ can be connected to $F$; we will call this the \emph{type} of a vertex. (2) A digraph from $\mathcal{A}$ is uniquely determined by a sequence, giving the number of vertices of each available type. (3) Any mapping between two such digraphs which fixes $F$ and respects types is a strong homomorphism.
Together with an application of Dickson's Lemma (which can be viewed as a special case of the well-known Higman's wqo theorem \cite{Hig52}) this will enable us to prove that $\mathcal{A}$ cannot be an antichain, and thus obtain a contradiction.

So, more formally, let $\mathbb{T}$ be the (finite) set of all digraphs obtained from $F$ by adding a single vertex to the vertex set and connecting it to $F$ arbitrarily. Let $\mathbb{T}=\{T_1, \ldots, T_P \}$.  We will refer to the elements of $\mathbb{T}$ as \emph{types}.  Let $D=F \cup D_e \cup D_c \in \mathcal{A}$ be arbitrary.  We shall say that a vertex $v \in D_e \cup D_c$ has \emph{type} $T_i$ ($1 \leq i \leq P$) if the subdigraph of $D$ induced on $F \cup \{v\}$ is isomorphic to $T_i$.

For $i=1,\ldots,P$, write:
\[ \tau_{e,i}(D)=|\{v \in D_e: v \mbox{ is of type } T_i \}|\]
\[ \tau_{c,i}(D)=|\{v \in D_c: v \mbox{ is of type } T_i \}|\]
and let
\[ \tau(D)=(\tau_{e,1}(D), \ldots, \tau_{e,P}(D); \tau_{c,1}(D), \ldots,\tau_{c.P}(D)).\]
We note that $D$ can be uniquely reconstructed from the sequence $\tau(D)$.

Dickson's Lemma now guarantees that the set of all $2P$-tuples of non-negative integers is wqo by the componentwise ordering.  Therefore there exist $D_1=F \cup D_{1,e} \cup D_{1,c}$ and $D_2=F \cup D_{2,e} \cup D_{2,c}$ in $\mathcal{A}$ such that
\[ \tau_{e,i}(D_1) \leq \tau_{e,i}(D_2) \mbox{ and } \tau_{c,i}(D_1) \leq \tau_{c,i}(D_2), i=1, \ldots,P. \]

In other words, if for $i \in \{1, \ldots,P \}$, $j \in \{1,2 \}$ and $z \in \{e,c\}$ we let $E_{i,j.z}$ be the set of vertices in $D_{j,z}$ of type $T_i$, then 
\[ |E_{i,1,z}| \leq | E_{i,2,z}| \mbox{ for all } i=1, \ldots,P; z \in \{e,c\}. \]
Furthermore, clearly
\[ D_{j,z}= \bigcup_{1 \leq i \leq P} E_{i,j,z}.\]

Let $\Phi: D_2 \rightarrow D_1$ be any mapping satisfying:
\begin{itemize}
\item[(1)] $\Phi|_F$ is the identity;
\item[(2)] $\Phi$ maps $E_{i,2,z}$ surjectively onto $E_{i,1,z}$.
\end{itemize}
We now prove that $\Phi$ is a strong homomorphism, which will contradict the fact that $\mathcal{A}$ is an antichain and complete the proof.  To see that $\Phi$ is a homomorphism we need to verify that it maps an arbitrary edge $(x,y)$ of $D_2$ onto an edge of $D_1$.  We have the following cases:
\begin{itemize}
\item If $x,y \in F$ then $(\Phi(x),\Phi(y))=(x,y) \in E(D_1)$.
\item If $x,y \in D_{e,2}$ then $x=y$ and $D_{e,1}, D_{e,2}$ are both reflexive.  Hence $\Phi(x)\in D_{e,1}$ and $(\Phi(x),\Phi(y))$ is a loop in $D_1$.
\item If $x,y \in D_{c,2}$ then $\Phi(x),\Phi(y) \in D_{c,1}$ which is complete, so $(\Phi(x),\Phi(y)) \in E(D_1)$.
\item If $x \in D_{2,z}$ and $y \in F$ then $x \in E_{i,2,z}$ for some $i$, so that $\Phi(x) \in E_{i,1,z}$ while $\Phi(y)=y$.  The pair $(\Phi(x),y)$ is an edge in $D_1$ because $x$ and $\Phi(x)$ have the same type.  The case when $x \in F$ and $ y \in D_{2,z}$ is analogous.
\item If $x \in D_{e,2}$ and $y \in D_{c,2}$ then $\Phi(x) \in D_{e,1}$ and $\Phi(y) \in D_{c,1}$; since all members of $\mathcal{A}$ have the same type of uniform connection between the empty and complete blocks, it follows that $(\Phi(x),\Phi(y)) \in E(D_1)$.  The case where $x \in D_{c,2}$ and $y \in D_{e,2}$ is analogous.
\end{itemize}

Finally, to prove that $\Phi$ is strong we need to show that for every edge $(x,y) \in E(D_1)$ there is an edge $(z,t) \in E(D_2)$ such that $(\Phi(z),\Phi(t))=(x,y)$.  This follows from the defining properties (1) and (2) of $\Phi$ and the assumption about uniform connections between the empty and complete components, via a case analysis similar to the above.
\end{proof}

Although the formulation of
Theorem \ref{HigmanDigraph} is somewhat technical,
it provides a general framework within which the wqo property for various classes defined by structural properties can be proved.
As a first example, we prove the following result.

For a digraph $D$, edges $(a_1,b_1), \ldots, (a_k,b_k)$ are said to be \emph{disjoint} if all the vertices $a_1,\ldots,a_k,b_1,\ldots,b_k$ are distinct, in which case  we refer to $\{(a_1,b_1), \ldots, (a_k,b_k)\}$ as a \emph{disjoint edge set}.

\begin{cor}\label{BoundedIndSet}
Any class of digraphs for which there is a uniform bound on the size of disjoint edge sets is well quasi-ordered under the standard and strong homomorphic image orderings.
\end{cor}
\begin{proof}
Let $N \in \mathbb{N}$ and let $\mathcal{C}$ be a class of digraphs such that all disjoint edge sets are of size $\leq N$.  We prove that $\mathcal{C} \subseteq \mathcal{T}_{2N}$, which is wqo by Theorem \ref{HigmanDigraph}.  Let $D \in \mathcal{C}$, and let $(a_1,b_1), \ldots,(a_k,b_k)$ be a maximal set of disjoint edges.  Observe that $k \leq N$.  Let $D_f=\{a_1,\ldots,a_k,b_1,\ldots,b_k \}$; clearly $|D_f|=2k \leq 2N$.  By maximality, every edge of $D$ has at least one of its endpoints in $D_f$.  Thus, letting $D_e$ comprise all the vertices of $D$ not in $D_f$, we see that the induced digraph on $D_e$ is empty.  Finally, setting $D_c=\emptyset$, we see that all the conditions from Theorem \ref{HigmanDigraph} are satisfied, proving that $D \in \mathcal{T}_{2N}$, as required.
\end{proof}

Recall that graphs can be viewed as (symmetric) digraphs, and so the above results can be specialised for graphs.
We record these specialisations for ease of future use:

\begin{thm}
\label{HigmanGraph}
Let $N$ be a natural number, and let $\mathcal{T}_{N}$ be the class of all reflexive graphs $G$ whose vertices can be split into a disjoint union $G_e\cup G_c\cup G_f$ such that:
\begin{itemize}
\item
the graph induced on $G_e$ is empty;
\item
the graph induced on $G_c$ is complete;
\item
$|G_f|\leq N$; 
\item
the connections between $G_e$ and $G_c$ are uniform, in the sense that for all $x,y\in G_e$ and all $z,t\in G_c$ 
we have that $x$ is adjacent to $z$ if and only if $y$ is adjacent to $t$.
\end{itemize}
The class $\mathcal{T}_{N}$ is well quasi-ordered under the standard and strong homomorphic image orderings.
\end{thm}

\begin{cor}\label{BoundedIndSetGraphs}
Any class of graphs for which there is a uniform bound on the size of disjoint edge sets is well quasi-ordered under the standard and strong homomorphic image orderings.
\end{cor}

\section{Graphs}
\label{sec-graphs}

In this section we will consider graphs in two possible models: irreflexive and reflexive, depending on the presence or otherwise of loops at individual vertices. Specifically, in a \emph{reflexive} graph it is assumed that such loops are present at \emph{all} vertices, while in the \emph{irreflexive} case there are \emph{none} at all. The class of all reflexive graphs will be denoted by $\Rgraphsclass$, while the class of all irreflexive graphs will be denoted by $\Igraphsclass$. While $\Rgraphsclass$ and $\Igraphsclass$ are equally valid models for the class of all graphs within the language of relational structures, choosing one of them profoundly affects the nature of homomorphisms.

For every $n\in\mathbb{N}$ we denote by $K_n$ the complete graph on $n$ vertices, using the same notation in both the reflexive and irreflexive cases.

\vspace{2mm}
\noindent
\textbf{Irreflexive graphs}
\vspace{2mm}

We begin with the class $\Igraphsclass$ where, in fact, we are able to solve the wqo problem completely.

\begin{thm}
\label{thm-Igraphs}
\begin{itemize}
\item[(1)]
A downward closed class $\mathcal{C}\subseteq\Igraphsclass$ of irreflexive graphs under the homomorphic image ordering is well quasi-ordered if and only if it is finite.
\item[(2)]
A class $\mathcal{C}\subseteq\Igraphsclass$ of irreflexive graphs defined by finitely many obstructions under the homomorphic image ordering or strong homomorphic image ordering is never well quasi-ordered.
\end{itemize}
\end{thm}

\begin{proof}
(1)
For the forward implication, suppose that $\mathcal{C}=\mathrm{Av}(O_i: i \in I)$ is wqo.
Since the complete graphs $\{K_1,K_2,\ldots \}$ form an antichain,
$\mathcal{C}$ can contain only finitely many of them. 
Hence, there must exist $m$ such that $K_j \not\in \mathcal{C}$ for all $j>m$.
So each $K_j$ with $j>m$ has homomorphic image $O_i$ for some $i\in I$.
However,
an irreflexive complete graph has no irreflexive homomorphic image other than itself, 
and so in fact $K_j=O_i$.
Hence the list of obstructions includes all $K_{m+1}, K_{m+2}, \ldots$.
Any graph $G$ has $K_{|G|}$ as a homomorphic image, and so  $G \not\in \mathcal{C}$ if $|G|>m$,
proving that $\mathcal{C}$ is finite.
The reverse direction is immediate.

(2) From the above, any wqo class under the standard homomorphic image ordering is finite and must contain infinitely many complete graphs in its obstruction set. 
Under the strong homomorphic image ordering, (1) no longer holds: for example, the family of all empty graphs is infinite and wqo.  However, it still remains true that a complete graph has no proper homomorphic image, and so again any class defined by finitely many obstructions necessarily contains all sufficiently large complete graphs.
\end{proof}

\vspace{2mm}
\noindent
\textbf{Reflexive graphs}
\vspace{2mm}

We now turn to reflexive graphs. 
For natural numbers $n$ and $k$ with $2k\leq n$, we define the \emph{subcomplete} graph $N_{n,k}$ to be the graph obtained from
the complete graph $K_n$ by deleting $k$ disjoint edges. More precisely, $N_{n,k}$ has vertices $\{1,\dots,n\}$ and edges
\[
\{\{i,j\}\::\: 1\leq i\leq j \leq n\} \setminus\{ \{1,2\},\{3,4\},\dots,\{2k-1,2k\}\};
\]
see Figure \ref{figNnk} for illustration.
The subcomplete graphs $N_{2k,k}$, in which every vertex participates in a non-edge, will be called \emph{proper}, while the remaining subcomplete graphs $N_{n,k}$ ($2k<n$) will be referred to as \emph{partial}.
The subcomplete graphs will play a dual role in what follows: the proper ones form an important antichain, while the partial ones will define precisely the wqo avoidance classes.
The key observation in establishing this is the following:

\begin{figure}

\begin{center}
\begin{tikzpicture}

\node (v1) at (120:1.2) {};
\node (v2) at (60:1.2) {} edge (v1);
\node (v3) at (0:1.2) {} edge (v1) edge (v2);
\node (v4) at (300:1.2) {} edge (v1) edge (v2) edge (v3);
\node (v5) at (240:1.2) {} edge (v1) edge (v2) edge (v3) edge (v4);
\node (v6) at (180:1.2) {} edge (v1) edge (v2) edge (v3) edge (v4) edge (v5);

\begin{scope}[radius=1mm]
\fill (v1) circle;
\fill (v2) circle;
\fill (v3) circle;
\fill (v4) circle;
\fill (v5) circle;
\fill (v6) circle;
\end{scope}

\end{tikzpicture}
\hspace{3mm}
\begin{tikzpicture}

\node (v1) at (120:1.2) {};
\node (v2) at (60:1.2) {};
\node (v3) at (0:1.2) {} edge (v1) edge (v2);
\node (v4) at (300:1.2) {} edge (v1) edge (v2) edge (v3);
\node (v5) at (240:1.2) {} edge (v1) edge (v2) edge (v3) edge (v4);
\node (v6) at (180:1.2) {} edge (v1) edge (v2) edge (v3) edge (v4) edge (v5);

\begin{scope}[radius=1mm]
\fill (v1) circle;
\fill (v2) circle;
\fill (v3) circle;
\fill (v4) circle;
\fill (v5) circle;
\fill (v6) circle;
\end{scope}

\end{tikzpicture}
\hspace{3mm}
\begin{tikzpicture}

\node (v1) at (120:1.2) {};
\node (v2) at (60:1.2) {};
\node (v3) at (0:1.2) {} edge (v1) edge (v2);
\node (v4) at (300:1.2) {} edge (v1) edge (v2);
\node (v5) at (240:1.2) {} edge (v1) edge (v2) edge (v3) edge (v4);
\node (v6) at (180:1.2) {} edge (v1) edge (v2) edge (v3) edge (v4) edge (v5);

\begin{scope}[radius=1mm]
\fill (v1) circle;
\fill (v2) circle;
\fill (v3) circle;
\fill (v4) circle;
\fill (v5) circle;
\fill (v6) circle;
\end{scope}

\end{tikzpicture}
\hspace{3mm}
\begin{tikzpicture}

\node (v1) at (120:1.2) {};
\node (v2) at (60:1.2) {};
\node (v3) at (0:1.2) {} edge (v1) edge (v2);
\node (v4) at (300:1.2) {} edge (v1) edge (v2);
\node (v5) at (240:1.2) {} edge (v1) edge (v2) edge (v3) edge (v4);
\node (v6) at (180:1.2) {} edge (v1) edge (v2) edge (v3) edge (v4);

\begin{scope}[radius=1mm]
\fill (v1) circle;
\fill (v2) circle;
\fill (v3) circle;
\fill (v4) circle;
\fill (v5) circle;
\fill (v6) circle;
\end{scope}

\end{tikzpicture}

\end{center}

\caption{The subcomplete graphs $N_{6,k}$ for $k=0,1,2,3$. Of these, $N_{6,0}\cong K_6$, $N_{6,1}$ and $N_{6,2}$ are partial, and $N_{6,3}$ is proper.}

\label{figNnk}

\end{figure}
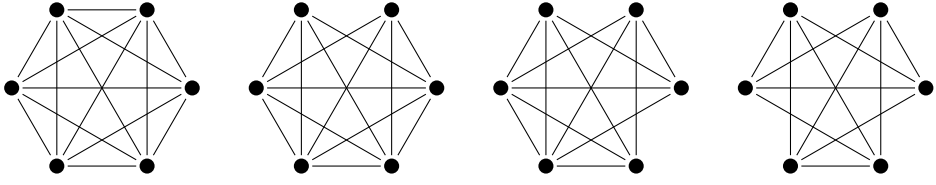

\begin{prop}
\label{prop-NC-hom-im}
Every proper homomorphic image of a subcomplete graph (proper or partial) is a partial subcomplete graph.
\end{prop}

\begin{proof}
Let $f:N_{n,k}\rightarrow H$ be a proper epimorphism from a subcomplete graph.
Since subcomplete graphs are characterised by the property that every vertex participates in at most one non-edge, and since this property is clearly preserved by homomorphisms, it follows that $H$ is again subcomplete.
To prove that $H$ is partial, let $p$ be any vertex of $H$ that has at least two preimages $u,v$ in $N_{n,k}$, let $q$ be any other vertex of $H$, and let $z$ be a preimage of $q$.
Since in $N_{n,k}$ the vertex $z$ participates in at most one non-edge, it follows that at least one of $\{u,z\}$ or $\{v,z\}$ is an edge, implying that $\{ p,q\}$ is an edge in $H$.
Hence the vertex $p$ participates in no non-edges in $H$.
\end{proof}

Note that the above proposition can be used for both standard and strong homomorphic image orderings.
An immediate consequence for both orderings is the following:

\begin{prop}
\label{propN2kk}
The family $\mathcal{N}=\{ N_{2k,k} \::\: k=1,2,3,\dots\}$ of all proper subcomplete graphs forms an antichain under the standard (and hence also strong) homomorphic image orderings.
\end{prop}

The first three members of the antichain $\mathcal{N}$ are shown in Figure \ref{figAntiN}.

\begin{figure}

\begin{center}

\begin{tikzpicture}[baseline=0mm]

\node (v1) at (90:1.1) {};
\node (v2) at (0:1.2) {};
\node (v3) at (270:1.2) {} edge (v1) edge (v2);
\node (v4) at (180:1.2) {} edge (v1) edge (v2);

\begin{scope}[radius=1mm]
\fill (v1) circle;
\fill (v2) circle;
\fill (v3) circle;
\fill (v4) circle;

\end{scope}

\end{tikzpicture}
\hspace{3mm}
\begin{tikzpicture}[baseline=0mm]

\node (v1) at (120:1.3) {};
\node (v2) at (60:1.3) {};
\node (v3) at (0:1.3) {} edge (v1) edge (v2);
\node (v4) at (300:1.3) {} edge (v1) edge (v2);
\node (v5) at (240:1.3) {} edge (v1) edge (v2) edge (v3) edge (v4);
\node (v6) at (180:1.3) {} edge (v1) edge (v2) edge (v3) edge (v4);

\begin{scope}[radius=1mm]
\fill (v1) circle;
\fill (v2) circle;
\fill (v3) circle;
\fill (v4) circle;
\fill (v5) circle;
\fill (v6) circle;
\end{scope}

\end{tikzpicture}
\hspace{3mm}
\begin{tikzpicture}[baseline=0mm]

\node (v1) at (135:1.5) {};
\node (v2) at (90:1.5) {};
\node (v3) at (45:1.5) {} edge (v1) edge (v2);
\node (v4) at (0:1.5) {} edge (v1) edge (v2);
\node (v5) at (315:1.5) {} edge (v1) edge (v2) edge (v3) edge (v4);
\node (v6) at (270:1.5) {} edge (v1) edge (v2) edge (v3) edge (v4);
\node (v7) at (225:1.5) {} edge (v1) edge (v2) edge (v3) edge (v5) edge (v6);
\node (v8) at (180:1.5) {} edge (v1) edge (v2) edge (v3) edge (v5) edge (v6);

\begin{scope}[radius=1mm]
\fill (v1) circle;
\fill (v2) circle;
\fill (v3) circle;
\fill (v4) circle;
\fill (v5) circle;
\fill (v6) circle;
\fill (v7) circle;
\fill (v8) circle;
\end{scope}

\end{tikzpicture}

\end{center}

\caption{The first three members $N_{4,2}, N_{6,3},N_{8,4}$ of the antichain $\mathcal{N}$.}

\label{figAntiN}

\end{figure}
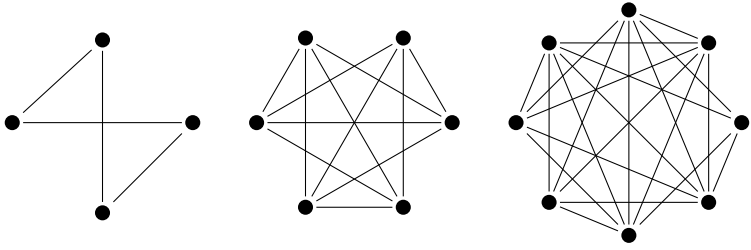

Now, specialising to the standard homomorphic image ordering, we have the following complete classification:

\begin{thm}
\label{thm-Rgraphs}
Let $G\in\Rgraphsclass$ be a reflexive graph.
The avoidance class $\Av(G)$ in $\Rgraphsclass$ under the homomorphic image ordering is well quasi-ordered if and only if $G$ is a partial subcomplete graph.
\end{thm}

\begin{proof}
($\Rightarrow$)
We prove the contrapositive. By Proposition \ref{prop-NC-hom-im}, if $G$ is not a partial subcomplete graph, then $G$ is not a homomorphic image of any $N_{2k,k}$ with $2k>|G|$. Hence, $\Av(G)$ contains all sufficiently large members of the antichain $\mathcal{N}$.

($\Leftarrow$)
Consider the avoidance class $\Av(N_{n,k})$ of a partial subcomplete graph.
By definition, $N_{n,k}$ possesses a set of $k$ disjoint non-edges (in the sense that no two non-edges share a vertex).
By properties of homomorphisms, any graph $H$ which has $N_{n,k}$ as a homomorphic image must also possess $k$ disjoint non-edges.
Conversely, every graph $H$ of size at least $n$,  which has a set of $k$ disjoint non-edges, can be mapped onto $N_{n,k}$:
simply map the $k$ non-edges of $H$ onto the $k$ non-edges of $N_{n,k}$, and map the remaining vertices of $H$ arbitrarily onto the vertices of degree $n-1$ in $N_{n,k}$, just making sure that surjectivity is satisfied. Therefore, $\Av(N_{n,k})$ can be expressed as $\mathcal{F}\cup\mathcal{D}$, where $\mathcal{F}$ is the finite
set comprising all graphs of size less than $n$, and $\mathcal{D}$ consists of all graphs whose sets of disjoint non-edges have size at most $k-1$.

We would now like to prove that $\mathcal{D}$ is wqo.
To this end, consider an arbitrary graph $H\in\mathcal{D}$.
Let $l$ ($\leq k-1$) be the the size of a maximal set of disjoint non-edges in $H$, and let $A=\{a_i,b_i\::\: i=1,\dots,l\}$
be the set of endpoints of these $l$ non-edges.
Then $H$ can be written as the disjoint union $A\cup (H\setminus A)$, where, clearly, the size of $A$ is bounded (by $2k-2$), while the subgraph induced on $H\setminus A$ is complete.
By Theorem \ref{HigmanGraph}, the collection of all graphs admitting such a decomposition is wqo.
Thus $\mathcal{D}$, and hence also $\Av(N_{n,k})$, is wqo as required.
\end{proof}

Finally, we turn to reflexive graphs under the strong homomorphic image ordering, where we do not have a complete picture and a potentially interesting open problem arises. First we observe that the proof of the forward direction in Theorem \ref{thm-Rgraphs} carries over verbatim:

\begin{thm}
\label{thm-RgraphsStrong}
Let $G\in\Rgraphsclass$ be a reflexive graph not isomorphic to any of the subcomplete graphs.
Then the avoidance class $\Av(G)$ in $\Rgraphsclass$ under the strong homomorphic image ordering is not well quasi-ordered.
\end{thm}

It is natural to ask whether all avoidance classes $\Av(N_{n,k})$ of subcomplete graphs are wqo under the strong homomorphic image ordering.   We can show that this is true in the case of complete graphs.

\begin{thm}
Within the class $\Rgraphsclass$ under the strong homomorphic image ordering,
the avoidance class $\Av(K_n)$ of a complete graph  is well quasi-ordered.
\end{thm}
\begin{proof}
Every mapping $\Phi:H \rightarrow K_n$, where $H$ is any graph, is a homomorphism.  Furthermore, if $H$ has a set of $\frac{n(n-1)}{2}$ disjoint edges, these can be mapped onto the edges of $K_n$, ensuring that $\Phi$ is a strong homomorphism, which in turn implies $H \not\in \Av(K_n)$.  Hence $\Av(K_n)$ is contained in the set of all graphs for which sets of disjoint edges have a bound of $\frac{n(n-1)}{2}$ on their size.  This class is wqo by Corollary \ref{BoundedIndSetGraphs}, as required.
\end{proof}

There remains the wqo question for $\Av(N_{n,k})$ with $k>0$.  The authors conjecture that $\Av(N_{n,1})$ is wqo for all $n$.  This is true for $n=3$: indeed, observe that a graph $H$ with $3$ or more vertices can be mapped onto $N_{3,1}$ by a strong homomorphism provided that the graph induced on its set of vertices of degree $\geq 1$ contains a non-edge. Hence $\Av(N_{3,1})$ is contained in the set of all graphs which are disjoint unions of one empty and one complete graph, a set which is wqo by Theorem \ref{HigmanGraph}.  
For $n=4$ we believe that an analogous, but more technical, analysis works.
For larger $n$ the situation for $\Av(N_{n,1})$ involves such increasing technical complications, that a different proof strategy might be needed.  The situation for $\Av(N_{n,k})$ remains open.

\section{Digraphs}
\label{sec-digraphs}

In this section, we establish characterisations of wqo classes within the class $\digraphsclass$ of all digraphs under the two homomorphic image orderings.  Furthermore, by analogy with the graph situation, we also consider the class $\Idigraphsclass$ of irreflexive digraphs  and the class of $\Rdigraphsclass$ of reflexive digraphs.

We begin by defining some distinguished families of digraphs.
For $n\geq 1$ we let $\overrightarrow{K}_n$ denote the \emph{complete digraph} on $n$ vertices; it has vertices $\{1,\dots,n\}$ and (directed) edges
$\{ (i,j) \::\: 1\leq i,j\leq n\}$.  Note that $\overrightarrow{K}_n$ is reflexive by definition.
Removing the loops $(i,i)$ ($1\leq i\leq n$) from $\overrightarrow{K}_n$ yields the 
\emph{irreflexive complete digraph} $\overrightarrow{K}_n^I$.
We denote by $\overrightarrow{\mathcal{K}}$ and $\overrightarrow{\mathcal{K}}_I$
the collections of all complete and complete irreflexive digraphs respectively; observe that the latter is an antichain, since a proper homomorphic image of any member must possess a loop.

By analogy with the subcomplete graphs $N_{n,k}$ from Section \ref{sec-graphs},
we define the \emph{subcomplete digraphs} as follows. For natural numbers $n$ and $k$ with $2k\leq n$, we let $\overrightarrow{N}_{n,k}$ be the digraph obtained from
the complete digraph $\overrightarrow{K}_n$ by taking $k$ disjoint (bidirectional) edges and making them uni-directional. More precisely, $\overrightarrow{N}_{n,k}$ has vertices $\{1,\dots,n\}$ and directed edges
\[
\{ (i,j) \::\: 1\leq i,j \leq n\} \setminus\{ (1,2),(3,4),\dots,(2k-1,2k)\};
\]
see Figure \ref{figNnkdigraph} for illustration.
The subcomplete digraphs $\overrightarrow{N}_{2k,k}$, in which every vertex participates in a uni-directional edge, will be called \emph{proper},
and we let
$\overrightarrow{\mathcal{N}}=\{ \overrightarrow{N}_{2k,k}\::\: k=1,2,3,\dots\} $.
The remaining subcomplete digraphs will be referred to as \emph{partial}.

\tikzset{->-/.style={decoration={
  markings,
  mark=at position #1 with {\arrow[scale=1.5]{>}}},postaction={decorate}}}

\tikzset{-<-/.style={decoration={
  markings,
  mark=at position #1 with {\arrow[scale=1.5]{<}}},postaction={decorate}}}

\begin{figure}

\begin{center}

\begin{tikzpicture}

\node (v1) at (120:1.5) {};
\node (v2) at (60:1.5) {};
\node (v3) at (0:1.5) {} edge (v1) edge (v2);
\node (v4) at (300:1.5) {} edge (v1) edge (v2);
\node (v5) at (240:1.5) {} edge (v1) edge (v2) edge (v3) edge (v4);
\node (v6) at (180:1.5) {} edge (v1) edge (v2) edge (v3) edge (v4) edge (v5);

\begin{scope}[radius=1mm]
\fill (v1) circle;
\fill (v2) circle;
\fill (v3) circle;
\fill (v4) circle;
\fill (v5) circle;
\fill (v6) circle;
\end{scope}

\end{tikzpicture}
\hspace{3mm}
\begin{tikzpicture}

\node (v1) at (120:1.5) {};
\node (v2) at (60:1.5) {} edge [thick,-<-=0.5] (v1);
\node (v3) at (0:1.5) {} edge [->-=0.65,-<-=0.4] (v1) edge [->-=0.7,-<-=0.3] (v2);
\node (v4) at (300:1.5) {} edge [->-=0.65,-<-=0.35] (v1) edge [->-=0.6,-<-=0.4] (v2) edge [thick,-<-=0.5] (v3);
\node (v5) at (240:1.5) {} edge [->-=0.6,-<-=0.4] (v1) edge [->-=0.65,-<-=0.35] (v2) edge [->-=0.65,-<-=0.4] (v3) edge [->-=0.7,-<-=0.3] (v4);
\node (v6) at (180:1.5) {} edge [->-=0.7,-<-=0.3] (v1) edge [->-=0.65,-<-=0.4] (v2) edge [->-=0.65,-<-=0.35] (v3) edge [->-=0.65,-<-=0.4] (v4) edge [->-=0.7,-<-=0.3] (v5);

\begin{scope}[radius=1mm]
\fill (v1) circle;
\fill (v2) circle;
\fill (v3) circle;
\fill (v4) circle;
\fill (v5) circle;
\fill (v6) circle;
\end{scope}

\end{tikzpicture}

\caption{The subcomplete graph $N_{6,2}$ and its directed counterpart $\protect\overrightarrow{N}_{6,2}$.
}

\label{figNnkdigraph}

\end{center}

\end{figure}
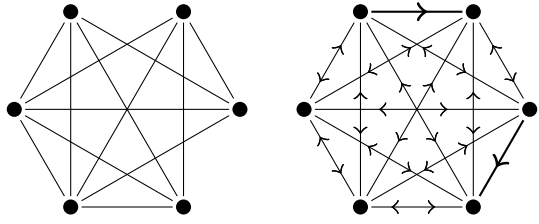

Parallelling Propositions \ref{prop-NC-hom-im} and \ref{propN2kk} we have:

\begin{prop}
\label{prop-NCD-hom-im}
Every proper homomorphic image of a subcomplete digraph (proper or partial) is a partial subcomplete digraph. In particular, the set $\overrightarrow{\mathcal{N}}$ of all proper subcomplete digraphs is an antichain (under both standard and strong homomorphic image orderings).
\end{prop}

\begin{proof}
The proof of Proposition \ref{prop-NC-hom-im} holds with `non-edge' replaced by `unidirectional edge' throughout.
\end{proof}

As in the case of graphs, this rapidly leads to the classification of wqo classes under both the standard and strong orderings:

\begin{thm}
\label{thm-digraphsPlain}
Let $D\in\digraphsclass$ be a digraph. The avoidance class $\Av(D)$ in $\digraphsclass$ under the  homomorphic image ordering is well quasi-ordered if and only if $D$ is complete, in which case $\Av(D)$ is finite.
\end{thm}

\begin{proof}
($\Rightarrow$) We show that if $D \in \digraphsclass$ is not complete then $\Av(D)$ is not wqo.

Suppose first that $D$ is not reflexive. 
By definition, 
$\overrightarrow{\mathcal{N}}$ consists of reflexive digraphs,
and is an antichain by Proposition \ref{prop-NCD-hom-im}. 
Since reflexivity is preserved by homomorphic images it follows that
$\overrightarrow{\mathcal{N}}\subseteq \Av(D)$, and so $\Av(D)$ is not wqo.

Next, suppose that $D$ is reflexive. Since it is not complete, there are distinct vertices $u,v$ such that
$(u,v)$ is not an edge in $D$.
Note that for any two distinct vertices $p,q$ in the complete irreflexive digraph $\overrightarrow{K}_n^I$
the pair $(p,q)$ is an edge, and that this property is preserved by epimorphisms.
It follows that the antichain $\overrightarrow{\mathcal{K}}_I$ is contained in $\Av(D)$, and hence $\Av(D)$ is not wqo, as required.

($\Leftarrow$) Let $D$ be a complete digraph.  Observe that any digraph with at least $|D|$ vertices can be mapped onto $D$ via a standard homomorphism.  Hence $\mathrm{Av}(D)=\{E: |E|<|D|\}$; this set is finite and therefore wqo.
\end{proof}

\begin{thm}
\label{thm-digraphsStrong}
Let $D\in\digraphsclass$ be a digraph. The avoidance class $\Av(D)$ in $\digraphsclass$ under the  strong homomorphic image ordering is well quasi-ordered if and only if $D$ is complete.
\end{thm}

\begin{proof}
($\Rightarrow$) 
This proof is identical to the corresponding direction in Theorem \ref{thm-digraphsPlain}.

($\Leftarrow$) Let us consider $\mathrm{Av}(\overrightarrow{K_n})$ for arbitrary $n$.  Every digraph $E$ with a set $(a_1,b_1),\ldots,(a_{n^2},b_{n^2})$ of $n^2$ disjoint edges can be mapped onto $\overrightarrow{K_n}$ via a strong homomorphism (any mapping $E \rightarrow \overrightarrow{K_n}$ which sends $(a_1,b_1), \ldots, (a_{n^2},b_{n^2})$ onto the $n^2$ directed edges of $\overrightarrow{K_n}$ is such a homomorphism).  Hence, the size of a maximal disjoint edge set in a member of $\mathrm{Av}(\overrightarrow{K_n})$ is uniformly bounded by $n^2$, implying that $\mathrm{Av}(\overrightarrow{K_n})$ is wqo by Corollary \ref{BoundedIndSet}.
\end{proof}

\vspace{2mm}
\noindent
\textbf{Irreflexive digraphs}
\vspace{2mm}

The situation with irreflexive digraphs is precisely analogous to the case of irreflexive graphs.  Specifically, the complete irreflexive digraphs have no proper homomorphic images, and every irreflexive digraph embeds into the irreflexive complete digraph of the same size under standard homomorphism.  This readily leads to the following characterisations:

\begin{thm}
\label{thm-Idigraphs}
\begin{itemize}
\item[(1)] A downward closed class $\mathcal{C}\subseteq\Idigraphsclass$ of irreflexive digraphs under the homomorphic image ordering is well quasi-ordered if and only if it is finite.
\item[(2)]
A class $\mathcal{C}\subseteq\Idigraphsclass$ of irreflexive digraphs defined by finitely many obstructions under either the standard or strong homomorphic image ordering is never well quasi-ordered.
\end{itemize}
\end{thm}

\vspace{2mm}
\noindent
\textbf{Reflexive digraphs}
\vspace{2mm}

We now turn our attention to \emph{reflexive digraphs}, i.e. those digraphs in which $(v,v)$ is an edge for every vertex $v$.  We will need another corollary of Theorem \ref{HigmanDigraph}, similar to Corollary \ref{BoundedIndSet}.  A pair of vertices $a,b$ in a digraph $D$ will be called \emph{partial}  if $a \neq b$ and $a \not \leftrightarrow b$ (i.e. if precisely one of the following holds: $a\rightarrow b$ or $a \leftarrow b$ or $a || b$).  Partial pairs $(a_1,b_1),\ldots,(a_k,b_k)$ are said to be \emph{disjoint} if all $a_1,\ldots,a_k,b_1,\ldots,b_k$ are distinct.

\begin{lemma}\label{BoundedPartialSet}
Any class of digraphs for which there is a uniform bound on the sizes of sets of disjoint partial pairs is well quasi-ordered.
\end{lemma}
\begin{proof}
Let $N \in \mathbb{N}$ and let $\mathcal{C}$ be a class of digraphs such that sets of disjoint partial pairs are all of size $\leq N$.  We prove that $\mathcal{C} \subseteq \mathcal{T}_{2N}$, which is wqo by Theorem \ref{HigmanDigraph}.  Let $D \in \mathcal{C}$, and let $(a_1,b_1), \ldots,(a_k,b_k)$ ($k \leq N$) be a maximal set of disjoint partial pairs in $D$.  Let $D_f=\{a_1,\ldots,a_k,b_1,\ldots,b_k \}$; clearly $|D_f|\leq 2N$.  By maximality, all pairs of vertices outside of $D$ are bidirectionally connected, so letting $D_c$ comprise all the vertices of $D$ not in $D_f$, and setting $D_e=\emptyset$, yields a decomposition of $D$ which meets the conditions from Theorem \ref{HigmanDigraph}.  Hence $D \in \mathcal{T}_{2N}$, as required.
\end{proof}

\begin{thm}
\label{thm-Rdigraphs}
Let $D\in\Rdigraphsclass$ be a reflexive digraph. The avoidance class $\Av(D)$ in $\Rdigraphsclass$ under the homomorphic image ordering is well quasi-ordered if and only if $D$ is a partial subcomplete digraph.
\end{thm}

\begin{proof}
($\Rightarrow$) This is a consequence of Proposition \ref{prop-NCD-hom-im}.  Indeed, if $\Av(D)$ is wqo, then $\Av(D) \cap \overrightarrow{\mathcal{N}}$ is finite.  Hence $D$ is a homomorphic image of infinitely many members of   $\overrightarrow{\mathcal{N}}$.  But proper homomorphic images of members of $\overrightarrow{\mathcal{N}}$ are precisely partial subcomplete digraphs. (This also holds for the strong order).

($\Leftarrow$) Let us consider $\Av(\overrightarrow{N}_{n,k})$ with $2k<n$.  Every digraph $E$ of size $\geq n$ with $k$ disjoint partial pairs $(a_1,b_1),\ldots,(a_k,b_k)$ can be homomorphically mapped onto $\overrightarrow{N}_{n,k}$.  To see this, map $(a_1,b_1),\ldots,(a_k,b_k)$ onto the $k$ unidirectional edges $(1,2), \ldots, (2k-1,2k)$ of $\overrightarrow{N}_{n,k}$ and extend to an epimorphism arbitrarily; note that this need not be strong.  Hence the size of a maximal set of disjoint partial pairs is bounded by $k$, and the result follows by Lemma \ref{BoundedPartialSet}.
\end{proof}

Finally, we turn to reflexive digraphs under the strong homomorphic image ordering, where, if one followed the parallel with graphs that is emerging, one would expect difficulties in determining the status of the classes of the form $\Av(\overrightarrow{N}_{n,k})$. Interestingly, this turns out not to be the case: 

\begin{thm}
\label{thm-RdigraphsStrong}
Let $D\in\Rdigraphsclass$ be a reflexive digraph.
Then the avoidance class $\Av(D)$ in $\Rdigraphsclass$ under the strong homomorphic image ordering is well quasi-ordered if and only if $D$ is a complete digraph.
\end{thm}

\begin{proof}
($\Rightarrow$) Suppose $\Av(D)$ is wqo.  As in the proof of Theorem \ref{thm-Rdigraphs}, $D$ must be a partial subcomplete digraph, say $D \cong \overrightarrow{N}_{n,k}$ with $2k<n$.  

To see that, in fact, $D$ must be complete, we need to consider another antichain.  Take the family $\mathcal{N}$ of proper subcomplete (undirected) graphs,
which is an antichain by Proposition \ref{propN2kk}, and view its members as directed graphs by interpreting every edge $\{a,b\}$ as a pair of directed edges $(a,b)$ and $(b,a)$.  The resulting collection of digraphs $\overrightarrow{\mathcal{N}^{\prime}}$ is an antichain, because identifying two distinct vertices via a homomorphism results in a vertex bidirectionally connected to all others.  Hence $\Av(D)$ can contain only finitely many members of $\overrightarrow{\mathcal{N}^{\prime}}$, implying that $D$ is a strong homomorphic image of infinitely many members of $\overrightarrow{\mathcal{N}^{\prime}}$.  Observe that, in any homomorphic image of a member of $\overrightarrow{\mathcal{N}^{\prime}}$, any pair of vertices $a,b$ is connected either bidirectionally or not at all.  This implies, in particular, that a subcomplete digraph  $\overrightarrow{N}_{n,k}$ can be such a homomorphic image only if it is complete, i.e. $k=0$.

($\Leftarrow$) The proof that $\Av(\overrightarrow{K_n})$ is wqo is identical to the corresponding part of the proof for general digraphs (Theorem \ref{thm-digraphsStrong}).
\end{proof}

We observe that the subcomplete digraphs, which have played a key role in this section, possess the property that their underlying graphs are complete, whereas the antichain  $\overrightarrow{\mathcal{N}^{\prime}}$ used in the proof of Theorem \ref{thm-RdigraphsStrong} does not possess this property.  It would be interesting to ask our wqo questions in the context of the class of all digraphs whose underlying graphs are complete.  This class includes all subcomplete digraphs and all tournaments.

\section{Tournaments}
\label{sec-tournaments}

In this section we consider the wqo problem for the class of tournaments.
As with graphs and digraphs, the choice of model affects which mappings qualify as homomorphisms. In fact, for tournaments, the irreflexive option is not particularly interesting: all homomorphisms are injective and the homomorphic image orderings reduce to equality. Also, for tournaments, homomorphisms and strong homomorphisms coincide.
So, we consider the class $\Rtournamentsclass$ of reflexive tournaments under the homomorphic image ordering.
It is perhaps mildly intriguing that even for this class the wqo problem can be completely solved in a trivial way: no class defined by finitely many obstructions is wqo under the homomorphic image ordering.

To prove this, we will require the following family of tournaments from \cite{HucRus15}. 
Let $n \in \mathbb{N}$ be odd.  The tournament $T_n$ on $n$ vertices $\{1,\ldots,n\}$ is given by the rule: for $1 \leq i < j \leq n$,
\[ \begin{array}{lcr} i \to j & \mbox{if} & i \not\equiv j \pmod 2,\\
j \to i & \mbox{if} & i \equiv j \pmod 2. \end{array} \]
The following result was observed in \cite{HucRus15}; for completeness we give a full proof here.

\begin{prop}
For any (odd) $n\geq 3$, the only homomorphic images of the tournament $T_n$ are $T_n$ itself and the one-element tournament.
\end{prop}

\begin{proof} 
For $n=3$, $T_n$ is a $3$-cycle, which clearly has no proper non-trivial homomorphic image.  So consider $n\geq 5$, and suppose $f: T_n \rightarrow T$ is a proper homomorphism onto a tournament $T$.  Note that directed triangles in $T_n$ correspond to triples $a<b<c$ with $a \not \equiv b$, $b \not \equiv c$ (and consequently $a \equiv c$),
where all congruences are being taken modulo $2$.  Let $i,j$ ($i<j$) be any two vertices of $T_n$ with $f(i)=f(j)=t$.  Suppose first that $i \equiv j$.  Then $\{i,i+1,j\}$ is a directed triangle and hence $f(i+1)=t$ as well.  Now, using the directed triangles $\{i,i+1,i+2\}, \{i+1,i+2,i+3\}, \ldots, \{n-2,n-1,n\}$ as well as $\{i-1,i,i+1\}, \{i-2,i-1,i\}, \ldots, \{1,2,3\}$, we see that in fact $f(x)=t$ for all vertices $x \in T_n$, and so $T$ is trivial.  If $i \not \equiv j$ then, because $n$ is odd, we have that $i \neq 1$ or $j \neq n$.  If $i \neq 1$ then $\{i-1,i,j\}$ is a directed triangle, yielding $f(i-1)=t$ and $i-1 \equiv j$, thus reducing to the previous case.  Similarly, if $j \neq n$, we have the directed triangle $\{i,j,j+1\}$, yielding $f(j+1)=t$ and $i \equiv j+1$.
\end{proof}

\begin{thm}
\label{thm-Rtournaments}
A class $\mathcal{C}\subseteq\Rtournamentsclass$ of reflexive tournaments defined by finitely many obstructions under the homomorphic image ordering is not well quasi-ordered.
\end{thm}
\begin{proof}
Let $\mathcal{C}=\mathrm{Av}(O_1,\ldots,O_k)$ and $n=\max(|O_1|,\dots,|O_k|)$.  
 Then no $T_m$ with $m > n$, can be mapped onto any of $O_1,\dots, O_k$. So 
$T_m \in \mathcal{C}$ for all $m > n$, and these form an antichain.
\end{proof}

\section{Concluding remarks}

We have obtained an almost complete characterisation of wqo classes of graphs, disgraphs and tournaments defined by a single obstruction under both standard and strong homomorphic image orderings.  One exception is the resolution of the wqo problem for the classes of reflexive graphs under the strong ordering defined by a single obstruction.
This, we have seen, hinges on deciding whether the avoidance classes $\Av(N_{n,k})$ of partial subcomplete graphs are wqo in general.

In fact, for irreflexive graphs, irreflexive digraphs and tournaments we have shown that no class defined by finitely many obstructions can be wqo under either ordering.
It seems that, in line with the situation for various subgraph orderings, the classification of all wqo classes defined by finitely many obstructions in the remaining settings is worth pursuing from both graph- and order-theoretic points of view.



\noindent
\textbf{Acknowledgement.}
The authors would like to thank the anonymous referees for their comments.

\end{document}